\documentclass[11pt]{amsart}

\usepackage[ansinew]{inputenc}    
\usepackage{graphicx}             

\newtheorem{theorem}{Theorem}
\theoremstyle{plain}

\newtheorem{corollary}{Corollary}

\newtheorem{lemma}{Lemma}

\newtheorem{proposition}{Proposition}
\newtheorem{remark}{Remark}

\numberwithin{equation}{section}

\begin{document}

\title[The trisector curve is transcendental]{The distance trisector curve is transcendental}

\thanks{This work is partially supported by grant MTM2009-08933 from the Spanish Ministry of Science and Innovation, and by CONACYT (Mexico), project 106 923.}

\author{J. Monterde, F. Ongay}
\address{J.Monterde. Dep. de Geometria i
Topologia, Universitat de Val\`encia, Avd. Vicent Andr\'es
Estell\'es, 1, E-46100-Burjassot (Val\`encia), Spain}
\email{monterde{@}uv.es}
\address{F. Ongay. CIMAT, Jalisco S/N, Valenciana, Gto., C.P.
36240, Mexico} \email{ongay@cimat.mx}



\date{January 30, 2013}

\begin{abstract}
We show that the distance trisector curve is not an algebraic
curve, as was conjectured in the founding paper by T. Asano, J.
Matousek and T. Tokoyama \cite{AMT}.
\end{abstract}

\maketitle

\section{Introduction}

In the paper \cite{AMT} (see also \cite{IK}) the curve called {\it
the distance trisector curve} was introduced. It is defined as
follows: Fix two points, say $P_1=(0,1)$ and $P_2=(0,-1)$; then
two curves, $C_1$, $C_2$, can be constructed such that they divide
the plane in three sectors---hence the designation of the curves
$C_1$, $C_2$ as ``trisectors,'' in such a way that the distance
between $C_1$ (the ``upper'' distance trisector curve) and $C_2$
(the ``lower'' curve) is the same as the distance from $C_2$ to
$P_2$; the curve $C_1$ then determines a connected region around
$P_1$ which can be thought of as the zone of influence of this
point, and similarly for $P_2$ and $C_2$, while the region between
the two curves is a kind of ``neutral zone,'' and the curves
$C_1$, $C_2$ are symmetric with respect to the $x$ axis.

Now, in the introduction of \cite{AMT} the authors conjecture that ``{\it \dots
the distance trisector curve is not
algebraic\dots}'' ---that is, that it cannot be expressed in the form
$P(x,y)=0$ for a polynomial $P\in{\mathbb R}[x,y]$, and this was probably the most important question left open in that work.

In this paper we show that this conjecture is true.  The sketch of our proof is as follows:

First, based on results in
\cite{AMT}, one knows that the distance trisector curve can be parametrized as
$(t,f(t))$ where $f(t)$ is an analytic function. The important point for us here is that
its Taylor expansion has coefficients in the quadratic field ${\mathbb Q}[\sqrt3]$, which in particular implies  that if the distance trisector curve were an
algebraic curve, expressed as $P(x,y)=0$, then
$P$ would belong to ${\mathbb Q}[\sqrt3][x,y]$.

The field ${\mathbb Q}[\sqrt3]$ has a natural involution: the conjugation map $\sqrt3\to -\sqrt3$ in ${\mathbb Q}[\sqrt3]$. Technically, this is the only non-trivial automorphism in the Galois group of the extension ${\mathbb Q}[\sqrt3]$ of ${\mathbb Q}$, although this is not essential to our proof; the point is  that from this another curve can be defined, which we call the {\it conjugate distance trisector
curve}, or simply, the {\it conjugate curve}. This new curve shares several of the algebro-geometric
properties of the distance trisector curve, but its geometric aspect is completely different as, roughly speaking,  the shape of the conjugate curve resembles that
of an Archimedean spiral.

But if the distance trisector curve were an algebraic curve
then the conjugate curve would also be, because it could be
expressed as $\overline{P}(x,y)=0$, where $\overline{P}\in {\mathbb Q}[\sqrt3][x,y]$ is the
conjugate polynomial. However, we can show that the
conjugate curve, like the Archimedean spiral,   has in fact an infinite number of crossings with the
axes. Therefore, it cannot be an algebraic curve.

\section{Some basic facts about the distance trisector curve}\label{section-envelope}

This section is  essentially an equivalent reformulation of some parts of \cite{AMT}.

\subsection{The distance trisector curve as an envelope curve...}
First of all, we observe that the distance trisector can be  seen as the envelope curve of
a family of circles: Indeed, if we
let $\alpha:I\to {\mathbb R}^2$ be a parameterization of the upper
distance trisector curve, and for each $t\in I$, we let $S_t$ be the
circle centered at $\alpha(t)$ and of radius $d(\alpha(t),(0,1))$, then the lower distance trisector curve is the envelope of the
family of circles $S_t$.
\medskip
\begin{center}
\includegraphics[width=10cm]{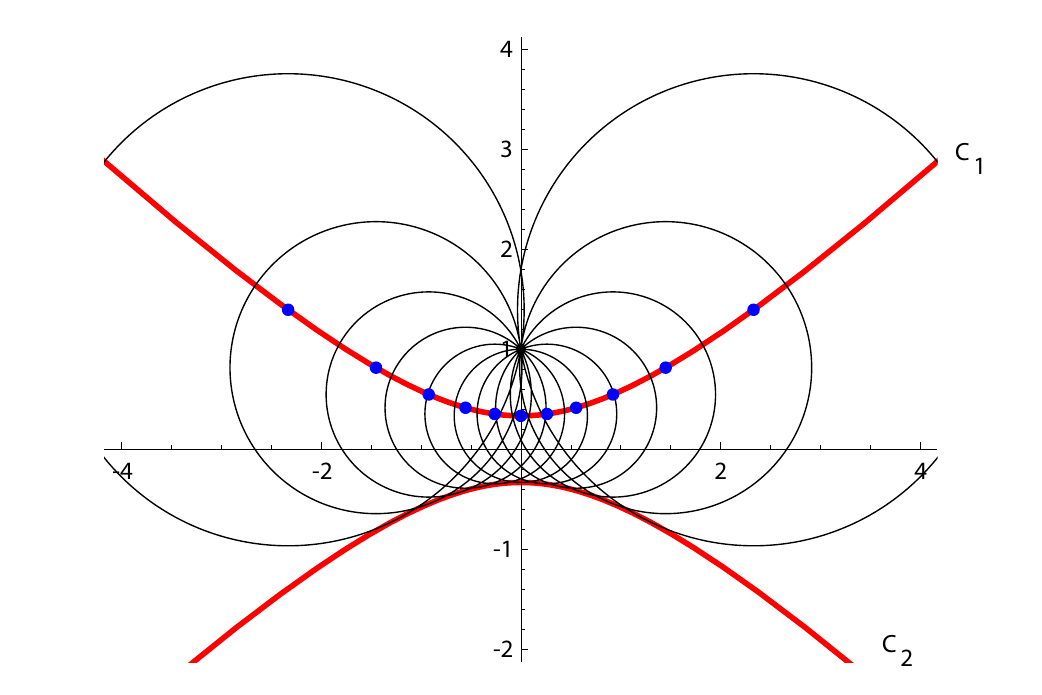}

Figure I. The two distance trisector curves and some of the circles $S_t$.
\end{center}
\medskip

For later use, let us next recall how to obtain in general the parametrization of the envelope curve of
a family of circles defined as before.

\begin{proposition}\label{prop-envelope}
Let $\alpha:I\to{\mathbb R}^2$, $I\subset {\mathbb R}$, be a
regular parametrized curve not passing through the point $(0,1)$.
For each $t\in I$, let $S_t$ be the circle centered at $\alpha(t)$
and of radius $d(\alpha(t),(0,1))$. The envelope curve of the
family of circles $\{S_t\}_{t\in I}$ can be parametrized by
$$\beta(t) = (0,1)+2\langle\alpha(t)-(0,1),\overrightarrow{\mathbf{n}}(t)\rangle\
\overrightarrow{\mathbf{n}}(t),$$ where
$\overrightarrow{\mathbf{n}}$ denotes the normal vector to the
curve $\alpha$.

Moreover, the tangent line to the envelope $\beta$ at $t_0$ only
depends on $\alpha(t_0)$ and on the tangent line to $\alpha$ at
$t_0$.
\end{proposition}

\begin{proof} Let us suppose that $\alpha(t) = (a(t),b(t))$,
then the implicit equation of the circle $S_t$ is
\begin{equation}\label{eq-circle}
F(x,y,t):=-1+x^2+y^2-2 x a(t)-2 (-1+y) b(t) = 0.
\end{equation}
Thus, if we solve the system of equations $$ \begin{cases}
F(x,y,t) &= -1+x^2+y^2-2 x a(t)-2 (-1+y) b(t)=0,\\ F_t(x,y,t) &=
-2 x a'(t)-2 (-1+y) b'(t) = 0, \end{cases}$$ in the variables $x,y$
we obtain a parameterization of the envelope of the family of
circles.

Indeed, from the second equation we have that $$\frac{x}{1-y} =
\frac{b'(t)}{a'(t)}.$$ Substituting now $ x
=(1-y)\frac{b'(t)}{a'(t)}$ into the first equation we get a
quadratic equation for $y$
$$ \frac{1-y}{(a'(t))^2}\left(-\left((a'(t))^2+(b'(t))^2\right) y-(a'(t))^2+(b'(t))^2-2a'(t)b'(t) a(t)+2  b(t)(a'(t))^2\right) = 0,$$
and it follows that
$$\begin{array}{rcl}
\beta(t) &=& \frac{1}{a'(t)^2+b'(t)^2} \left( 2 b'(t) ((1-b(t))
a'(t)+a(t) b'(t)),(2 b(t)-1) a'(t)^2-2 a(t) a'(t)
b'(t)+b'(t)^2\right)\\[3mm]&=& \frac{1}{||\alpha'(t)||^2}
\left( 2((b(t)-1)a'(t)-a(t)
b'(t))(-b(t),a(t))+(0,||\alpha'(t)||^2)\right)\\[3mm]
&=&
2\langle(a(t),b(t)-1),\frac{(-b(t),a(t))}{||\alpha'(t)||}\rangle\frac{(-b(t),a(t))}{||\alpha'(t)||}+(0,1)\\[3mm]
&=&(0,1)+2\langle\alpha(t)-(0,1),\overrightarrow{\mathbf{n}}(t)\rangle\
\overrightarrow{\mathbf{n}}(t),
\end{array}$$
as stated.
\bigskip

For the second assertion, we simply compute the tangent vector to
the curve $\beta$, and the last expression shows that this is
independent of the parameterization  of the initial curve
$\alpha$; to simplify matters, it is better to work with the
arc-length parameterization; then, a simple computation  shows
that
$$\beta'(s)=-2\kappa(s)\left(\langle\alpha(s)-(0,1),\overrightarrow{\mathbf{t}}(s)\rangle
\overrightarrow{\mathbf{n}}(s)+\langle\alpha(s)-(0,1),\overrightarrow{\mathbf{n}}(s)\rangle
\overrightarrow{\mathbf{t}}(s) \right),$$ where
$\overrightarrow{\mathbf{t}}$ denotes the tangent vector to the
curve $\alpha$ and $\kappa$ its curvature function.
\end{proof}

\begin{remark}\label{remark-tangent-beta}
Later on we will need  the unit  tangent vector to the curve
$\beta$. Therefore, let us compute first $||\beta'||$ {\rm :}
$$\begin{array}{rcl}
||\beta'(s)||^2 &=&
4\kappa^2(s)\left(\langle\alpha(s)-(0,1),\overrightarrow{\mathbf{t}}(s)\rangle^2+
\langle\alpha(s)-(0,1),\overrightarrow{\mathbf{n}}(s)\rangle^2\right)\\[3mm]
&=& 4\kappa^2(s)||\alpha(s)-(0,1)||^2. \end{array}$$ Thus,
$$\overrightarrow{\mathbf{t}}^\beta(s) = -\langle\frac{\alpha(s)-(0,1)}{||\alpha(s)-(0,1)||},\overrightarrow{\mathbf{t}}(s)\rangle
\overrightarrow{\mathbf{n}}(s)-\langle\frac{\alpha(s)-(0,1)}{||\alpha(s)-(0,1)||},\overrightarrow{\mathbf{n}}(s)\rangle
\overrightarrow{\mathbf{t}}(s).$$
\end{remark}

\subsection{Comparison with Lemma 8 in Assano {\it et al.}} The important point about stating the previous proposition  is that it leads to consider a map $\Theta$ that, given a parametrized curve
$\alpha(t) =(a(t),b(t))$, transforms it to another parametrized
curve
$$\Theta(\alpha)(t)  =
2\langle\alpha(t)-(0,1),J\left(\frac{\alpha'(t)}{||\alpha'(t)||}\right)\rangle\
(T\circ J)\left(\frac{\alpha'(t)}{||\alpha'(t)||}\right)-(0,1),$$
 where $T$ is the symmetry
$T(x,y) = (x,-y)$ and $J$ is the rotation $J(x,y) = (-y,x)$.
It is actually plain  that, by definition, the distance trisector curve is characterized by the
following property: it is a curve such that if $\alpha(t)$
is a local parametrization then $\Theta(\alpha)(t)$ is another local
parametrization of the same curve; that is
$$\Theta(\alpha)(t) =\alpha(g(t)),\qquad \forall t\in{\mathbb R},$$
with $g:{\mathbb R}\to {\mathbb R}$ the reparametrization.

To facilitate comparisons with \cite{AMT}, we now observe that in that work the trisector curve is described as the graph of a function $f$ defined on the whole real line ${\mathbb R}$,
$\alpha(x) = (x,f(x))$. In other words, the parameter $t$ is chosen to be $x$, and the coordinate functions $a(t)\to x$ and $b(t)\to f(x)$; according to the previous paragraph, this can be written as
$$\Theta(\alpha)(x)=(t(x),f(t(x))),$$
where the reparametrization $g$ has been denoted by $t$.

Let us next recall the following basic result from \cite{AMT}:

\begin{lemma} (Lemma 8 of \cite{AMT}) The following equations are satisfied for every $x\in{\mathbb R}$:
\begin{equation}\label{eq-lemma-8}
\begin{array}{rcl}
(t(x)-x)^2+(f(t(x))+f(x))^2-x^2-(f(x)-1)^2&=&0,\qquad{\rm and}\\[3mm]
t(x)-x+\left(f(x)+f(t(x))\right)f'(t(x)) &=&0,
\end{array}
\end{equation}
where $f'(t(x))$ is the derivative of $f$ evaluated at $t(x)$.
\end{lemma}

\begin{remark}
The first equation in (\ref{eq-lemma-8}) is just $F(t(x),-f(t(x)),x)=0$ (see Eq. (\ref{eq-circle}) for the definition of $F$).
The second equation in (\ref{eq-lemma-8}) comes from the fact that for a fixed $x$, the point $(t(x),-f(t(x)))$ minimizes the squared distance of $(x,f(x))$ to
$(u,-f(u))$ among all $u$. Therefore,
$$0=\frac12\frac{\partial\ }{\partial u}|_{u=t(x)}\left((u-x)^2+(f(x)+f(u))^2 \right) =t(x)-x+(f(x)+f(t(x)))f'(t(x)).$$
By the way, it is perhaps worthwhile of notice that in \cite{AMT} this equation appears without two parentheses. We believe this to be a misprint, later corrected in Eq. (4) of the paper.
\end{remark}

\section{The conjugate distance trisector curve}

In \cite{AMT} the authors are looking for a {\bf convex}
curve, because intuitivelly this is the shape of the
distance trisector curve in a neighborhood of its initial point $(0,1/3)$.
Nevertheless, the convexity hypothesis is not essential for the algebraic manipulations, and can be suppressed.
Somewhat surprisingly, another curve appears
sharing with the distance trisector curve many of its properties; we have called this new curve the conjugate distance trisector curve, or for brevity, the {\it conjugate\/} curve.
Let us now elaborate on this point:

\subsection{Power series for $f$ and $t$ near the
origin}\label{series}

Along the proof on Lemma 10 in \cite{AMT}, which is the technical result needed to compute the Taylor series expansion of the distance trisector curve,  the authors arrive to a
point where a solution for the equations
\begin{equation}\label{equation-lambda}
\lambda-1+\frac43 \lambda q_2=0,\qquad \lambda^2-2\lambda+\frac43 (\lambda^2+2) q_2=0,
\end{equation}
has to be found; here $\lambda$ and $q_2$ are the unknowns.
Now, equations (\ref{equation-lambda}) are easily seen to have the two different sets of solutions:
\begin{equation}\label{two-solutions}
\begin{cases}
\lambda &=+\sqrt3-1,\qquad q_2 =\frac38(+\sqrt3-1),\\[4mm]
\lambda &= -\sqrt3-1,\qquad q_2 = \frac38(-\sqrt3-1),
\end{cases}
\end{equation}
and notice that the only difference between them is the sign of
$\sqrt3$.

In any case, choosing one of these solutions, the rest of the coefficients in the
Taylor series expansion can be obtained recursively, as solutions
of linear systems where all the coefficients belongs to the field
${\mathbb Q}[\sqrt3]$. In particular, all the coefficients in the
Taylor expansion belong to this field.
The choice made in \cite{AMT}  is then the first set of solutions, which gives raise to a convex curve.

But the second set of solutions is of course also possible, and this gives raise to a concave solution, which in a sense is a conjugation of the previous curve.
In other words, its Taylor
expansion is essentially  the same as the Taylor expansion of the distance
trisector curve, but with $\sqrt3$ replaced by $-\sqrt3$. Therefore the proof of the convergence of the new series is completely analogous to the
existing one. (For instance, in this case the determinant $d_k$ of the matrix of
coefficients of the linear system is always negative, and the
maximum value is $d_4 = -497.415$, so the minimum absolute value
is $497.415$.) The first few terms in the Taylor expansions for the new solution are in fact
$$f(x) = \frac{1}{3}-\frac{3}{8} \left(1+\sqrt{3}\right)x^2-\frac{27}{704} \left(13+7 \sqrt{3}\right)x^4+O(x^6),$$
$$t(x) = -(1+\sqrt{3})x+\frac{27}{88} \left(17+10 \sqrt{3}\right)x^3+O(x^5).$$

\subsection{Extending to all of ${\mathbb R}$}\label{extending-the-curve}
The main difference between the distance
trisector curve and its conjugate curve  is that the latter is not the graph of a function, as this curve has self-intersections.

To justify this, we notice that although in principle the functions $f$ and $t$ for the conjugate curve are known to be analytic only on some neighborhood of $0$, in fact Lemma 11 in \cite{AMT} can again be used to extend this
neighborhood iteratively.

These assertions can be nicely illustrated using an approximation to the conjugate curve ---obtained form Proposition \ref{prop-envelope}--- as follows: take $\alpha_0(t)=\left(t,\frac13-t^2\right)$, $t\in[-1,1]$, and  define recursively $\alpha_{i+1}=\Theta(\alpha_i)$ for $i>0$. It turns out that $\alpha_5(t)$  already gives a very good approximation in an extended interval to the conjugate curve, and as shown in Fig. II, it has a shape resembling an Archimedean spiral.

\begin{center}
\includegraphics[width=7.5cm]{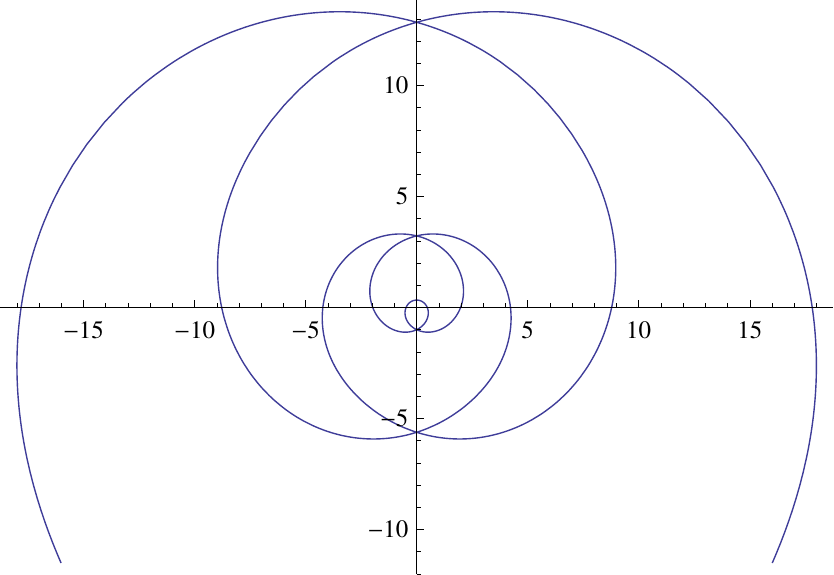}

Figure II. The conjugate distance trisector curve.
\end{center}

\section{Some properties of the conjugate distance trisector curve.}

\subsection{The reflected curve} As  said in Section \ref{section-envelope}, the lower distance
trisector curve can be seen as the
envelope curve of a family of circles centered at points of the upper curve.
 Due
to the fact that it has been defined through Eqs.
(\ref{eq-lemma-8}), the same happens to the conjugate distance trisector curve
 (see Fig. III).

\medskip

\begin{center}
\includegraphics[width=8cm]{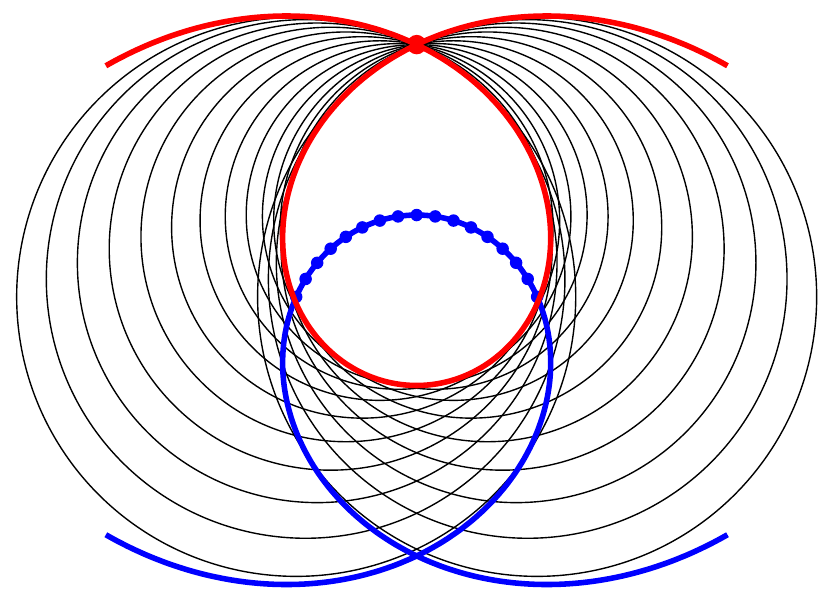}

Figure III. The conjugate distance trisector curve (blue), its reflected
curve (red) and some of the circles $S_t$.
\end{center}

In this case it makes little sense to speak of ``upper'' and ``lower'' curves, and thus we will rather refer to them as the conjugate  curve and its reflection
(with respect to the $x$-axis).

This result also shows that for any point $\alpha(t)$ of the conjugate
distance trisector curve there is indeed a point $\beta(t)=T(\Theta(\alpha)(t))$ of the reflected
curve such that the distance between $\alpha(t)$ and $\beta(t)$ is the same as the distance between $\alpha(t)$ and
$p=(0,1)$.

Moreover, this brings up another important difference between the two curves; namely,  that if $\alpha(t_0)$ is a point of the conjugate curve and $\beta(t_0)=\alpha(g(t_0))$ is its corresponding point on the reflected curve, then we can only assure that this point {\it locally\/} minimizes the squared distance from $\alpha(t_0)$ to $\beta(t)$.

As an explicit example, take $t_0=\frac1{32}$, so that $\alpha_5(t_0) = (0.92795, 2.82373)$ and $\Theta(\alpha_5)(t_0) = (2.2336, -4.39928) =\alpha_5(-0.0858323)$.

\begin{center}
\includegraphics[width=10cm]{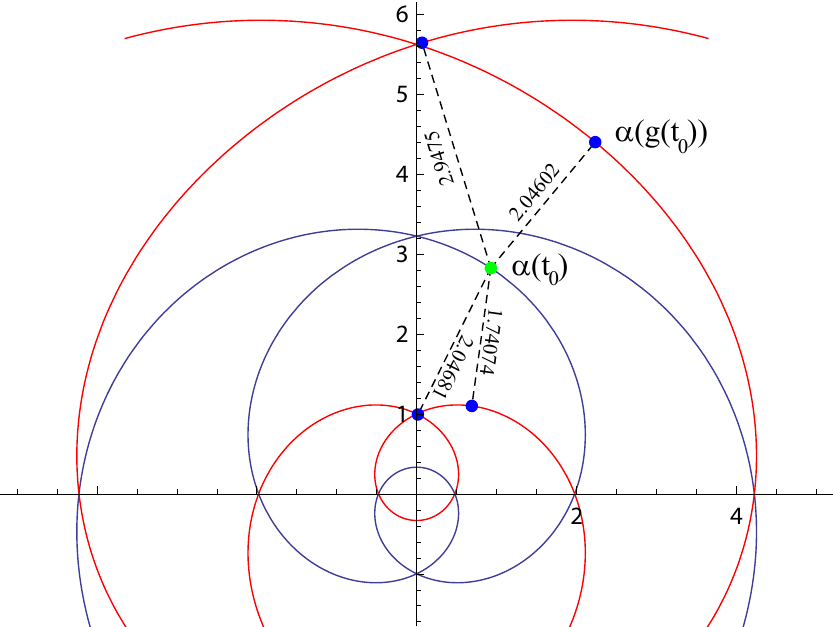}

Figure IV. Visualization of the first four local minima of the distance function $||\Theta(\alpha_5)(\frac1{32})-\alpha_5(t)||$.
\end{center}

The function $||(2.2336, 4.39928)-\alpha_5(t)||^2$ then has a local minimum ($4.18942$) at $t=-0.0858323$, but it has a global minimum  ($3.03018$) at $t=0.0134386$ (see Fig. IV).

\begin{remark}  As can be seen from Fig. IV, it is only the point marked $\alpha(g(t_0))$ that satisfies the property that its distance to the conjugate curve is the same as the distance from this curve to the point $(0,1)$, and thus this is the point that belongs to the circle $S_{t_0}$ of the envelope construction.

The slight discrepancies between the distances (less than $1/1000$) reflect the fact that $\alpha_5$ is only an approximation to the conjugate curve, but also show that it is indeed a very good one.
\end{remark}


Another property of the distance trisector curve, of perhaps still more interest here, is the following
easy consequence of Proposition \ref{prop-envelope} (see also \cite{AMT}).

\begin{corollary}\label{coro-envelope}
If $\alpha$ is a parametrization of the distance trisector curve,
then the segment joining the point $(0,1)$ with
$T(\Theta(\alpha)(t))$ is parallel to the normal vector to $\alpha$
at $\alpha(t)$.
\end{corollary}

A graphical interpretation of the statement of Corollary
\ref{coro-envelope} is given in Fig. V, and it is then clear that the same property is still valid for the conjugate curve (see Fig. VI).

\begin{center}
\includegraphics[width=8 cm]{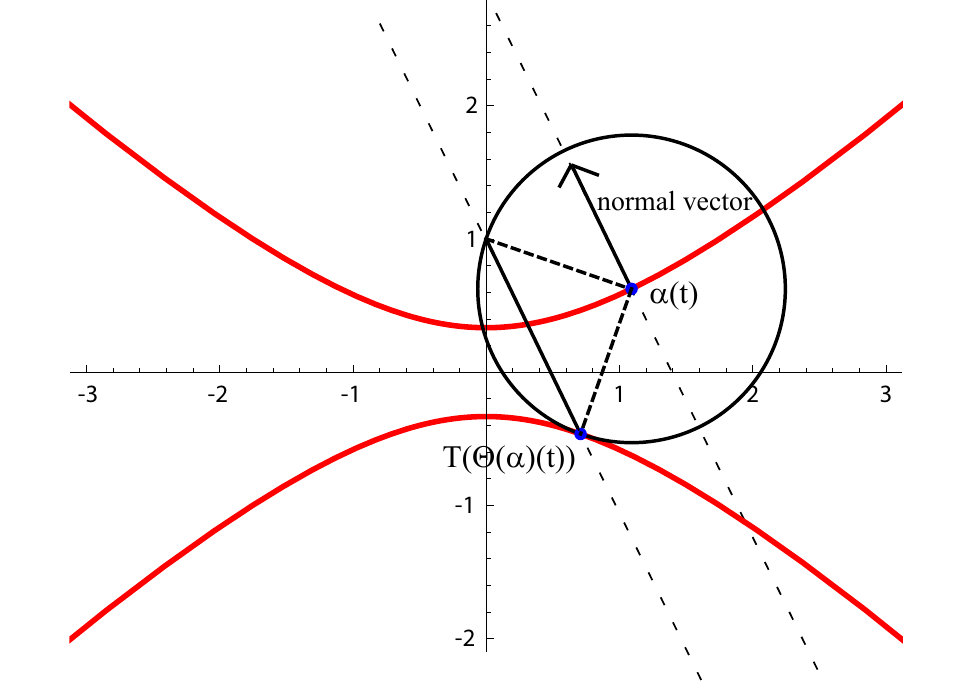}

Figure V.  A graphical interpretation of the statement of Corollary
\ref{coro-envelope}.
\end{center}
\medskip

\begin{center}
\includegraphics[width=8cm]{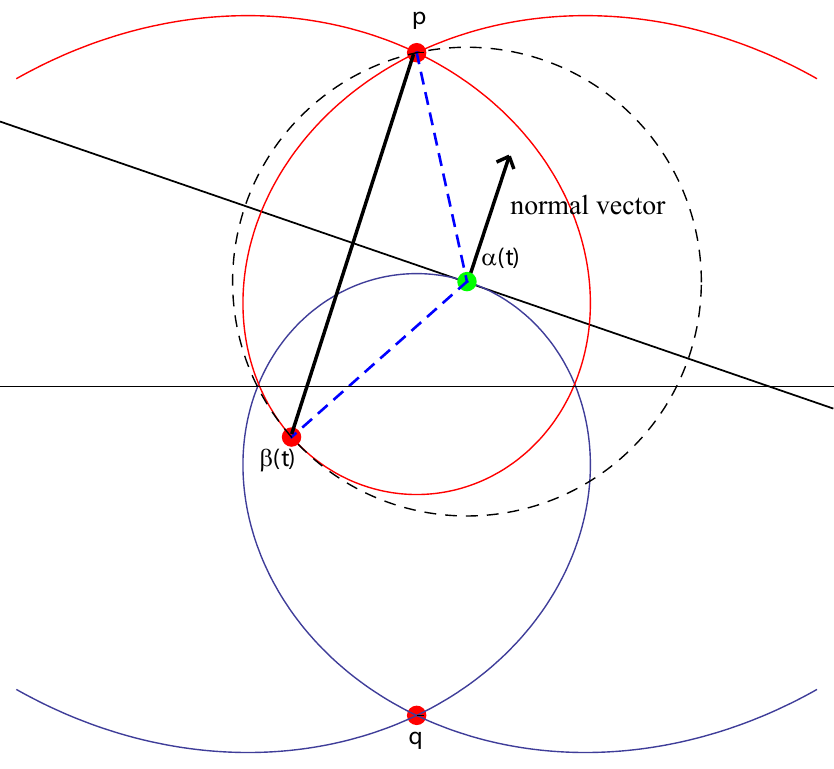}

Figure VI. The conjugate distance trisector curve (blue) and its horizontal
reflection (red) also satisfy Corollary \ref{coro-envelope}.
\end{center}

An  important point to stress here is that the tangent line at
$\alpha(t)$ is the bisectrix of the angle
$\widehat{P\alpha(t)\beta(t)}$.

\subsection{The conjugate curve has self-intersections} Let us now  prove the existence of the crossings in  the conjugate curve:

\begin{proposition}The conjugate distance trisector curve passes through $(0,-1)$. Therefore it has a self-intersection at this point.
\end{proposition}\label{cruce}

\begin{proof} It is clear that  at the point $(0,\frac13)$ the tangent line to the
conjugate distance trisector curve is horizontal. However, and in contrast to the trisector curve, it is not difficult to see,
even for the local Taylor series expansion, that there is a point where
the tangent line is vertical. Indeed, numerical computations show that the first point where the tangent
line is vertical is approximately $(0.524251, -0.243883)$.

Now, as the tangent line runs, either to the right or to the left,
from the horizontal position to the vertical position, there is a point where it passes
through the point $(0,1)$.

\begin{center}
\includegraphics[width=7cm]{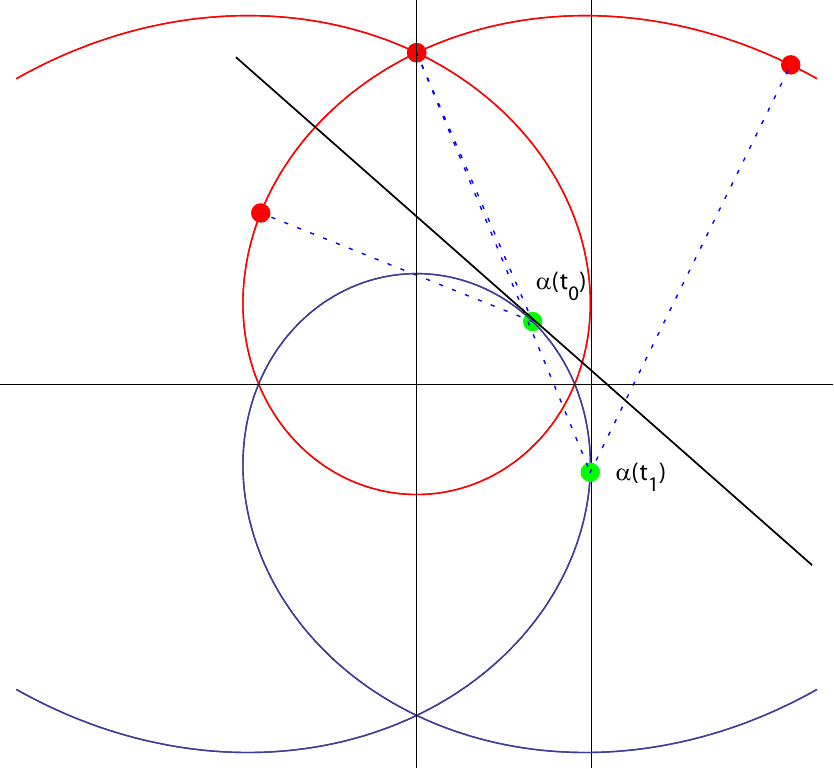}

Figure VII. When the tangent line at $\alpha(t_0)$ moves to the tangent line
at $\alpha(t_1)$, there is a point where it passes through the
point $(0,1)$.
\end{center}

 Such a point exists because if we asign, say, a positive value to the angle  $\widehat{P\alpha(t)\beta(t)}$ at a point such as $\alpha(t_0)$, then the angle at $\alpha(t_1)$ is negative and therefore at some point it has the value $0$ (see Fig. VII).
\bigskip

\begin{center}
\includegraphics[width=6cm]{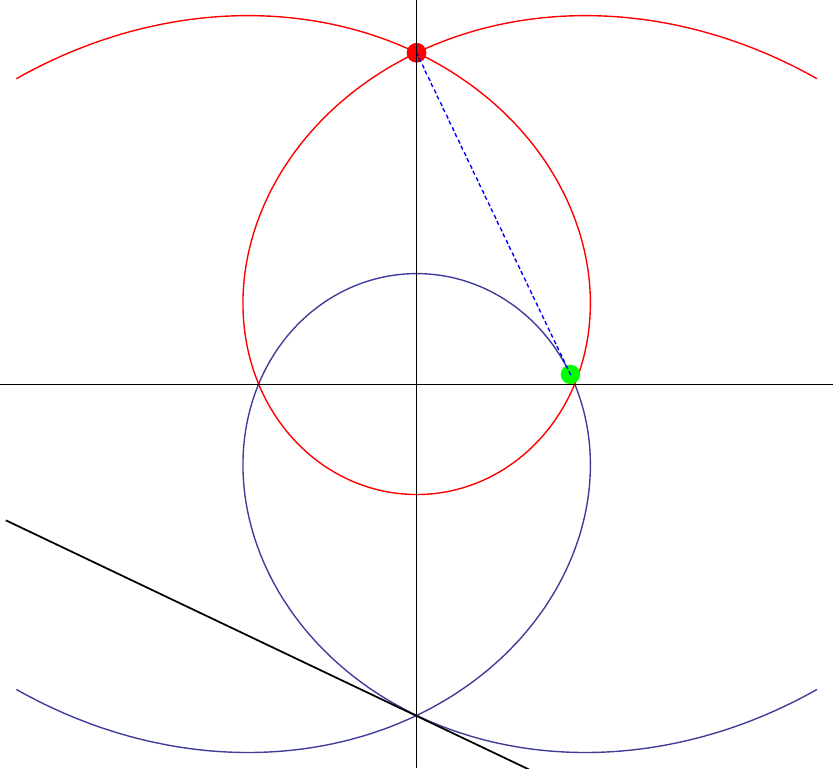}

Figure VIII. The tangent line at $(0.464045, 0.0289289)$ passes through  $(0,1)$.
\end{center}
\medskip

But if the reflected curve passes through $(0,1)$, the
conjugate curve passes through $(0,-1)$. Since the curve is symmetric with respect to $y$ axis, this point is a self-intersection of the curve.
\end{proof}

\begin{remark}\label{tangent-at-(0,-1)}
Later on we will need a tangent vector at $(0,-1)$. A  numerical approximation shows that one of  the tangents at this point is generated by the unit vector $(0.902272, -0.431168)$ (see Fig. VIII).
\end{remark}

\subsection{Horizontal tangent lines of
the conjugate  curve}

More generally, horizontal tangent lines are associated to
crossings of the conjugate  curve with the $y$-axis:
\begin{lemma}\label{lema-horizontal}
Suppose $\alpha(t_0)$ is a point on the conjugate distance trisector
curve with horizontal tangent line; then $T(\beta(t_0))$ is a
point on the $y$-axis, whose tangent line is parallel to
$J(\alpha(t_0)-(0,1))$.
\end{lemma}

\begin{proof} Recall that $J$ is rotation through an angle of $\pi/2$, and $T$ reflection with respect to the $x$ axis. Let us also recall that the tangent vector to the curve
$$\beta(t) = (0,1)+2\langle\alpha(t)-(0,1),\overrightarrow{\mathbf{n}}(t)\rangle\;
\overrightarrow{\mathbf{n}}(t),$$ is given by
$$\overrightarrow{\mathbf{t}}^\beta(t) = -\langle\frac{\alpha(t)-(0,1)}{||\alpha(t)-(0,1)||},\overrightarrow{\mathbf{t}}(t)\rangle\;
\overrightarrow{\mathbf{n}}(t)-\langle\frac{\alpha(t)-(0,1)}{||\alpha(t)-(0,1)||},\overrightarrow{\mathbf{n}}(t)\rangle\;
\overrightarrow{\mathbf{t}}(t).$$

Now, if $\alpha(t_0)$ is a point on the conjugate distance trisector
curve with horizontal tangent line, then
$$\overrightarrow{\mathbf{t}}(t_0)=J(\overrightarrow{\mathbf{t}}(t_0))=(1,0),\quad
\overrightarrow{\mathbf{n}}(t_0)=(0,1).$$
Thus,
$$\beta(t_0) =(0,1)+2\langle(x(t_0),y(t_0)-1),(0,1)\rangle\; (0,1)= (0,2y(t_0)-1),$$
and therefore, $T(\beta(t_0))=(0,-2y(t_0)+1)$ is a point where the conjugate curve crosses the $y$ axis. We will refer to this point as the {\it  crossing point associated\/} to $\alpha(t_0)$
\bigskip

Moreover, if we write
$$\frac{\alpha(t_0)-(0,1)}{||\alpha(t_0)-(0,1)||}=(a,b),$$
then
$$\overrightarrow{\mathbf{t}}^\beta(t_0) = -(b,a),$$
and
$$T(\overrightarrow{\mathbf{t}}^\beta(t_0)) = (-b,a) =J(a,b).$$
\end{proof}

\begin{center}
\includegraphics[width=10cm]{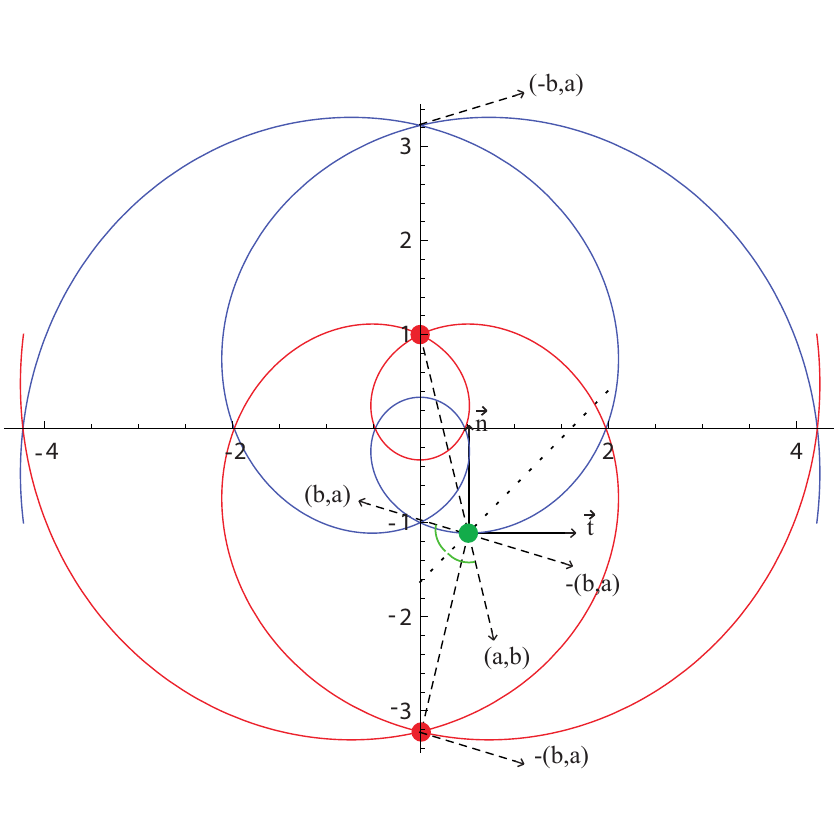}

Figure IX. Visualization of the statement of Lemma \ref{lema-horizontal}.
\end{center}

A consequence of this is that when $\alpha(t_0)$ is below the line $y=-1$,
then the associated crossing point is above the line $y=1$ and one of the
tangent vectors to the conjugate curve at that point has both coordinates strictly positive, as shown in Fig. IX.
Clearly, an analogous result holds true when $\alpha(t_0)$ is over the line $y=1$. In this case,
the associated crossing point is under the $x$-axis and a
tangent vector can be taken with both coordinates strictly negative.

\section{If the distance trisector curve were an algebraic curve...}

 We have seen in Section \ref{series} that the distance trisector
curve can be parametrized as $\alpha(t) = (t,f(t))$, where the
 function $f$ is analytic and its Taylor
expansion has coefficients in
${\mathbb Q}[\sqrt3]$. To make the connection with the aim of this paper we now prove:

\begin{lemma}
If the distance trisector curve were an algebraic curve,  defined by an implicit equation $P(x,y)=0$
with $P\in{\mathbb R}[x,y]$, then we can assume that
$P\in{\mathbb Q}[\sqrt3][x,y]$.
\end{lemma}

\begin{proof}

We can  suppose that $P$ is an irreducible polynomial of
total degree $\le n$.
Also,  for any non
zero real number $a$,  the implicit equation $(aP)(x,y)=0$ defines the same algebraic curve; but, by irreducibility of $P$, if we fix any non vanishing coefficient in the
polynomial, then there is just one possible implicit equation.
\bigskip

Also, since $\alpha(0) = (0,\frac13) = (0,f(0))$, it is better to
write the polynomial $P$ in terms of the power basis
$\{x^i(y-\frac13)^j\}_{i,j}$:
$$P(x,y) = \sum_{i+j\le
n}\frac{p_{i,j}}{i!j!}x^i\left(y-\frac13\right)^j,\qquad
p_{i,j}\in{\mathbb R}.$$ We will now show that all the coefficients
$p_{i,j}$ can be chosen in ${\mathbb Q}[\sqrt3]$.
\bigskip

Let us write the Taylor expansion of $f$ as
$$
f(t) = \sum_{i\in{\mathbb N}} m_i t^i=\frac13+m_1t+m_2t^2+\dots,
$$
where $m_i\in {\mathbb Q}[\sqrt3]$.
Furthermore, since $f$ is obviously an even function, we can suppose that $m_{2k+1}=0$ for $k\in{\mathbb N}$.
\bigskip

From the equation $P(x,y)=0$ we have that $P(t,f(t))
=0$. Therefore, the derivatives also satisfy
$\frac{d^k\ }{dt^k}{|}_{t=0}P(t,f(t)) =0$, for any $k\in{\mathbb N}$.
\bigskip

The first derivative is
$$
\frac{d\ }{dt}|_{t=0}P(x(t),y(t))= P_x(0,\frac13)+P_y(0,\frac13)f'(0)=P_x(0,\frac13)=p_{1,0},
$$
since $f'(0)=m_1=0$. Therefore, $p_{1,0} = 0$.

\medskip

The second derivative is
$$\begin{array}{rcl}
\frac{d^2\ }{dt^2}|_{t=0}P(t,f(t))&=& P_{xx}(0,\frac13)+ 2 P_{xy}(0,\frac13)f'(0)+ P_{yy}(0,\frac13)(f'(0))^2+ P_y(0,\frac13)f''(0)\\[3mm]
&=&p_{2,0}+ p_{0,1}m_2,
\end{array}
$$
and so on.

Thus, in general the condition $\frac{d^k\ }{dt^k}{|}_{t=0}P(x(t),y(t))
=0$ can be written as a homogeneous linear equation in the
unknowns $p_{i,j}$, where the coefficients are computed from $m_i$
through sums or products; and since we have supposed that the curve is
algebraic,  there are solutions to all these equations. We can moreover suppose that
there is a coefficient of $P$  equal to $1$, for if $p_{i_0,j_0}\ne 0$,
then the $i_0,j_0$ coefficient of $\frac{P}{p_{i_0,j_0}}$ is $1$.
But assuming this, the solution is unique, because the polynomial $P$ is irreducible.

Since all the coefficients in
the system belong to ${\mathbb Q}[\sqrt3]$, the same holds
for its solution $\{p_{i,j}\}_{0\le i+j\le n}$. Thus, $P\in{\mathbb
Q}[\sqrt3][x,y]$, as stated.
\end{proof}

\bigskip

Now,  the conjugation map:
$a+b\sqrt3\to a-b\sqrt3$ in the  field  ${\mathbb
Q}[\sqrt3]=\{a+b\sqrt3\ |\ a,b\in{\mathbb Q}\}$ extends to
the polynomial ring ${\mathbb Q}[\sqrt3][x,y]$.
And if an algebraic curve is defined by an equation $P(x,y) = 0$ with $P\in {\mathbb Q}[\sqrt3][x,y]$, then
the conjugate polynomial, $\overline{P}$, also defines an algebraic curve.  Therefore and applying this to our case, if the distance trisector curve were an algebraic curve, its conjugate curve would be algebraic too.

\section{The conjugate distance trisector curve is not an algebraic curve}

The main technical result is now:

\begin{lemma}\label{inifinite-crossings}
There is an infinite number of intersections between the $y$-axis and the conjugate distance trisector curve.
\end{lemma}

\begin{proof}
Let $\alpha:{\mathbb R}\to {\mathbb R}^2$, $\alpha(t) =
(x(t),y(t))$, be a regular parametrization of the conjugate
curve. We are going to construct a sequence
$\{C_n=\alpha(t_n)=(0,y(t_n))\}_{n\in{\mathbb N}}$ of crossing
points such that
$$t_0 = 0, \qquad C_0=\alpha(0) = (0,\frac13),$$
$$t_1 > 0, \qquad C_1=\alpha(1) = (0,-1),$$
$$t_n<t_{n+1},\qquad 0\notin x(]t_n,t_{n+1}[),$$
$${\rm sg}(y(t_n)) = (-1)^n,\qquad |y(t_n)| <|y(t_{n+1})|,\qquad
{\rm and}\ |y(t_n)-1|\ge 2^{n-1}\ (n>0).$$

We will use some auxiliary sequences:
A sequence
$$\{V_n=\alpha(s_n)=(x(s_n),y(s_n))\}_{n\in{\mathbb N}}$$ of
points with vertical tangent line ---that is, $x'(s_n)=0$, such that
$$ x(s_n)<x(s_{n+1}),\qquad x(s_n)>2^n,$$
and a sequence
$$\{H_n=\alpha(u_n)=(x(u_n),y(u_n))\}_{n\in{\mathbb N}}$$ of
points with horizontal tangent line ---that is $y'(u_n)=0$, such that
$$|y(u_n)|<|y(u_{n+1})|,\qquad |y(u_n)|>2^n,$$
and such that
$$t_n<u_n<s_n<t_{n+1}.$$
Finally, we will need a sequence $\{P_n=\alpha(r_n)\}$ of points where the conjugate curve crosses the line $y=1$.

Throughout the proof, points on the reflected curve will be marked with a\ $\ \widetilde{ }\ $, whereas the corresponding points on the conjugate curve will go without the\ $\ \widetilde{ }\ $.

 Obviously, $H_0=(0,1/3)$, and the existence of the points $V_0=(x(s_0),y(s_0))$ and $P_0=C_1=(0,-1)$ was established in Proposition \ref{cruce}
\medskip

Now, since $V_0$ is between $C_0$ and $C_1$, then
$$t_0=0<s_0<t_1.$$ And from what has been said, such
a point is related to another point in the reflected curve, $\widetilde{P}_1$,
with second coordinate $=1$ (see Fig. X, left).

\begin{center}
\includegraphics[width=14cm]{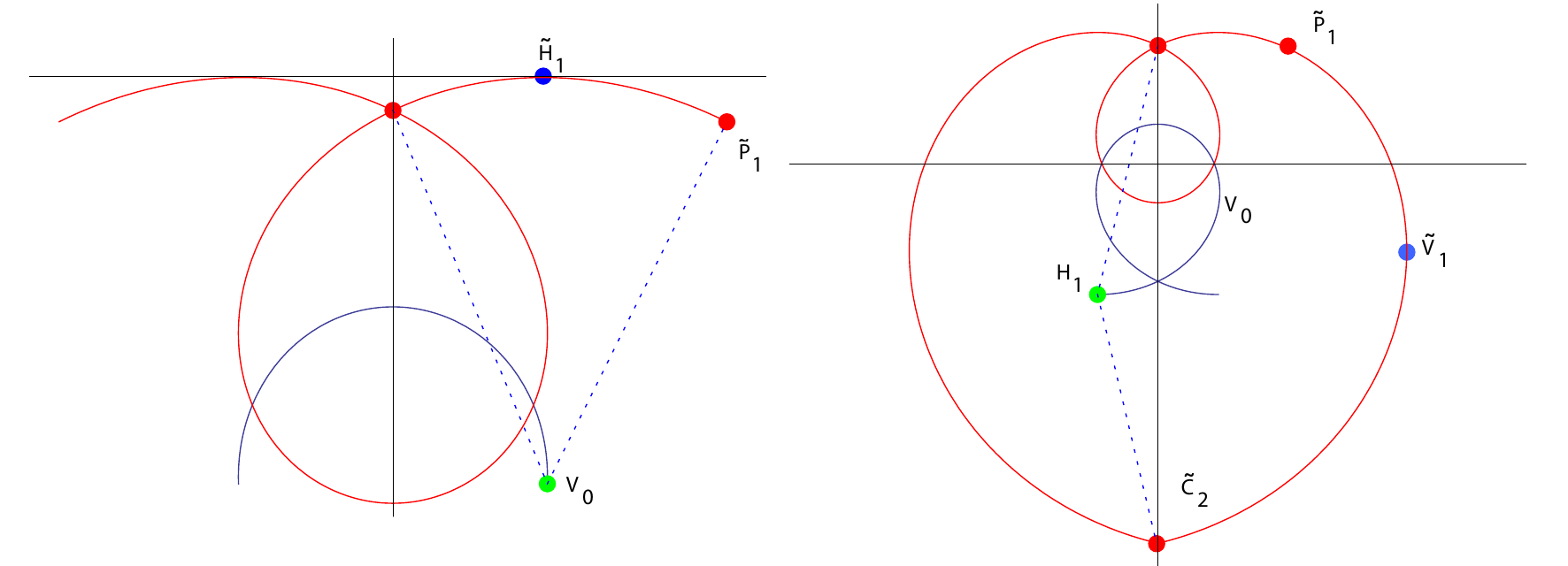}

Figure X. Construction of the first points where the tangent lines are horizontal or vertical .
\end{center}

Therefore, between $(0,1)$ and $\widetilde{P}_1$ there should be a point, $\widetilde{H}_1=(x(u_0),y(u_0))$, in
the reflected curve with horizontal tangent, and similarly for the
conjugate curve. Let us call the latter $H_1$. Notice that $t_1<u_1$  and  that the absolute
value of the second coordinate of $\widetilde{H}_1$ is greater than $1$, because
at $(0,1)$ the reflected curve has a tangent vector with second coordinate positive (see Remark \ref{tangent-at-(0,-1)}).

Since $H_1$ has an horizontal tangent line, it has associated a new
crossing point in the reflected curve $\widetilde{C}_2=(0,-y(t_2))$, and Proposition \ref{cruce} also gives that $-y(t_2)< -3$. Therefore $|y(t_2)-1|=y(t_2)-1>3-1=2$.

Now, between $\widetilde{P}_1$ and $\widetilde{C}_2$, there is a point
on the reflected curve with vertical tangent, $\widetilde{V}_1=(x(s_1),y(s_1))$, with $u_1<s_1<t_2$.
And by a similar reasoning as the one in Proposition \ref{cruce}, its first coordinate is greater than twice the first coordinate of $V_0$.
(see Fig. X, right).

We can then iterate the process: $\widetilde{V}_1$ generated the crossing $\widetilde{C}_2$, but because of the envelope construction of the reflected curve, $V_1$ also generates a point $\widetilde{P}_2$; and between $\widetilde{V}_1$ and $\widetilde{P}_2$, there is also a point in the
reflected curve with horizontal tangent, $\widetilde{H}_2$, so we have the corresponding points in the conjugate curve, and so on, as shown in Fig. XI:

\begin{center}
\includegraphics[width=14cm]{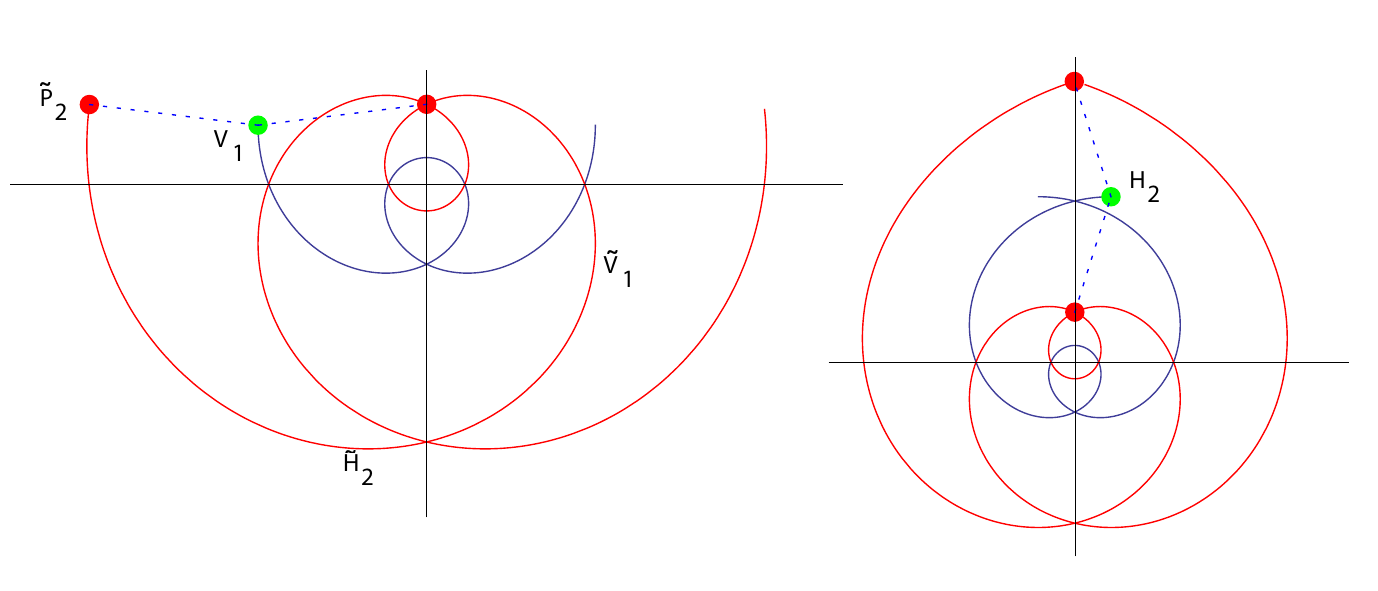}

Figure XI. Construction of the next points where the tangent lines are horizontal or vertical .

\end{center}


Let us now state our induction hypothesis;
we will suppose that $y(t_n)>0$, the other case being analogous:

Between $C_{n-1}=(0,y(t_{n-1}))$ and $C_n=(0,y(t_n))$ there is a point $V_{n-1}=(x(s_{n-1}),y(s_{n-1}))$ with $t_{n-1}<s_{n-1}<t_{n}$. This point generates another
point $\widetilde{P}_n$ in the reflected curve, whose second coordinate is less
than the second coordinate of $C_n$.

\begin{center}
\includegraphics[width=8cm]{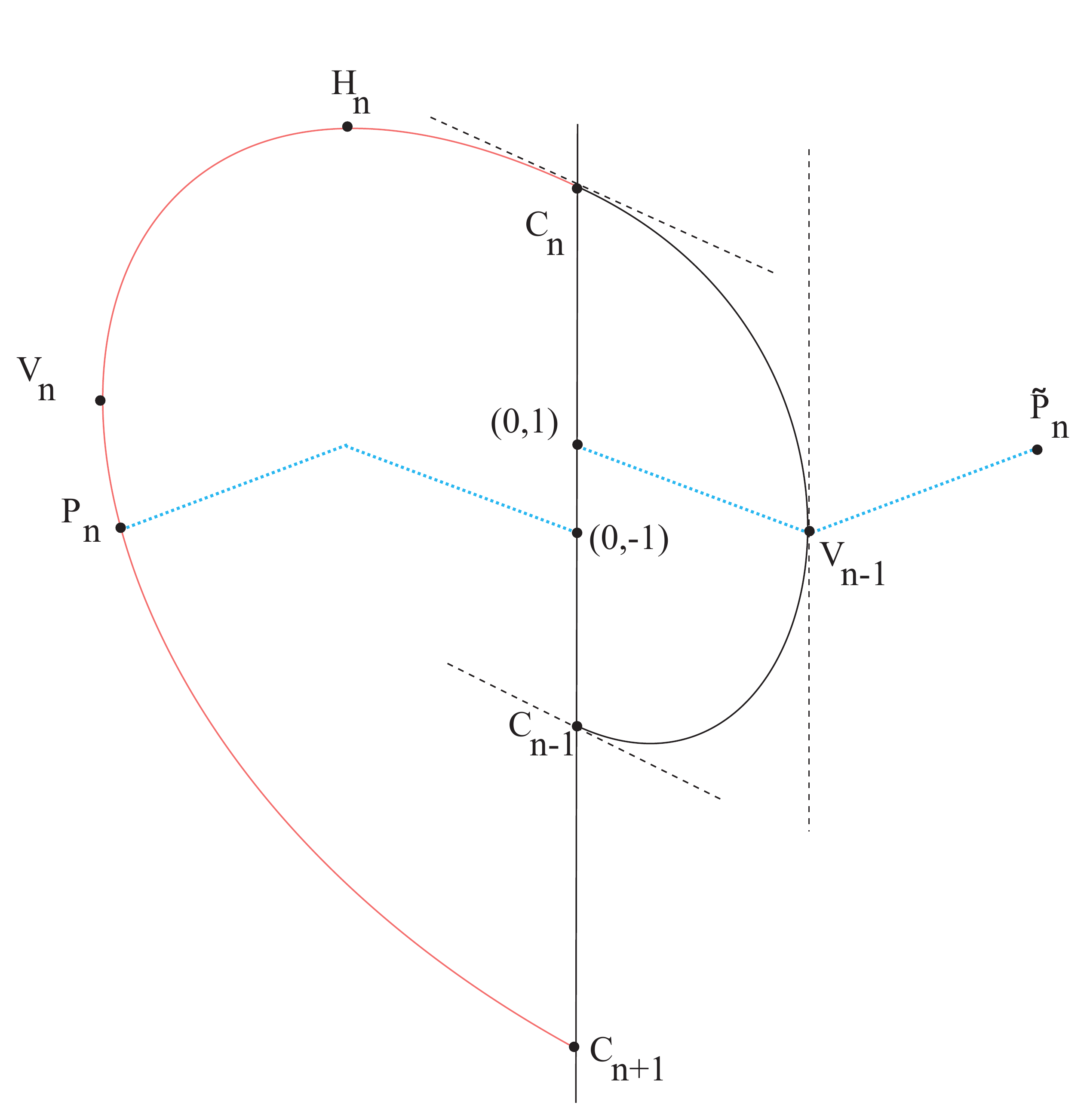}

Figure XII. Inductive proof of the existence of the crossing point $C_{n+1}$.
\end{center}

The point $\widetilde{P}_n$ then defines a point $P_n$ in the conjugate
curve. Since the conjugate curve goes from $C_n$ to $P_n$, and
since the tangent line at $C_n$ goes up, then there is a
new point $H_n=(x(u_n),y(u_n))$ with
$y(t_n)<y(u_n)$ and $t_n<u_n$.
This new point defines another new crossing point
$\widetilde{C}_{n+1}=(0,-y(t_{n+1}))$ in the reflected curve, and thereby a point
$C_{n+1}=(0,y(t_{n+1}))$ in the conjugate curve, and so on.

The fact now is
that, just as in Proposition \ref{cruce}, $-y(t_{n+1}) = 1+2(y(u_n)-1) > 1+2(y(t_n)-1)=2y(t_n)-1$.
Thus, $$|y(t_{n+1})-1|= -y(t_{n+1})+1 >2y(t_n)-1+1 >2\times 2^{n-1} = 2^n.$$

Finally, since the conjugate curve goes from $H_n$ to
$C_{n+1}$, then there is a new point $V_n=(x(s_n),y(s_n))$, with $u_n<s_n<t_{n+1}$, whose
distance to the $y$-axis is greater than the distance from $P_n$
to the $y$-axis. This distance is greater than twice the distance from $V_{n-1}$ to
the $y$-axis.

In particular, all the crossing points are distinct.
\end{proof}

We are now ready to conclude:

\begin{theorem}
The distance trisector curve is a transcendental curve.
\end{theorem}

\begin{proof}
As already mentioned, if the distance trisector curve were an algebraic curve, then its
conjugate curve would be an algebraic curve too.

But since the number
of intersections of the conjugate curve with the  $y$-axis is
infinite, this curve cannot be algebraic because, according the to Bézout's Theorem,
the number of intersections between any two algebraic curves is always finite.

Therefore, the distance trisector curve is transcendental, as claimed.
\end{proof}

\end{document}